\documentclass{amsart}

\usepackage[latin1]{inputenc}

\usepackage[vcentermath]{youngtab}
\usepackage{amsmath}
\usepackage{amsthm} 
\usepackage{verbatim}

\newtheorem{Thm}{Theorem}[section]
\newtheorem{Cor}[Thm]{Corollary}
\newtheorem{Lem}[Thm]{Lemma}
\newtheorem{Prop}[Thm]{Proposition}
\newtheorem{Property}[Thm]{Property}

\newtheorem{Def}[Thm]{Definition}
\newtheorem{Rem}[Thm]{Remark}
\newtheorem{Ex}[Thm]{Example}

\newcommand{\ie}  {\textit{i.e.,\ }}

\newcommand{\bx} {\mathbf{x}}
\newcommand{\bT} {\mathbb{T}}
\newcommand{\bS} {\mathbb{S}}
\newcommand{\bR} {\mathbb{R}}
\newcommand{\bZ} {\mathbb{Z}}
\newcommand{\naturals}  {\mathbb{N}}

\newcommand{\from} {\leftarrow}

\def\charge{ {\rm {charge}}}

\title{Combinatorial expansions in $K$-theoretic bases}

\author{Jason Bandlow}
\address{Department of Mathematics, University of Pennsylvania,
Philadelphia, PA 19104}
\email{jbandlow@math.upenn.edu}
\author{Jennifer Morse}
\thanks{Research supported in part by NSF grants
DMS:1001898,0652641,0638625}
\address{Department of Mathematics, Drexel University,
Philadelphia, PA 19104}
\email{morsej@math.drexel.edu}

\begin{document}

\begin{abstract}
We study the class $\mathcal C$ of symmetric functions whose
coefficients in the Schur basis can be described by generating
functions for sets of tableaux with fixed shape.  Included in this
class are the Hall-Littlewood polynomials, $k$-atoms, and Stanley
symmetric functions; functions whose Schur coefficients encode
combinatorial, representation theoretic and geometric information.
While Schur functions represent the cohomology of the Grassmannian
variety of $GL_n$, Grothendieck functions $\{G_\lambda\}$ represent
the $K$-theory of the same space.  In this paper, we give a
combinatorial description of the coefficients when any element of
$\mathcal C$ is expanded in the $G$-basis or the basis dual to
$\{G_\lambda\}$.
\end{abstract}

\maketitle

\section{Introduction} \label{sec:intro}

Schubert calculus uses intersection theory to convert enumerative
problems in projective geometry into computations in cohomology rings.
In turn, the representation of Schubert classes by Schur polynomials
enables such computations to be carried out explicitly.  The
combinatorial theory of Schur functions is central in the application
of Schubert calculus to problems in geometry, representation theory,
and physics.

In a similar spirit, a family of power series called Grothendieck
polynomials were introduced by Lascoux and Sch\"utzenberger in
\cite{LS82} to explicitly access the $K$-theory of $GL_n/B$.  In
\cite{FK94}, Fomin and Kirillov first studied the stable limit of
Grothendieck polynomials as $n \to \infty$.  When indexed by
Grassmannian elements, we call these limits the $G$-functions.
Grothendieck polynomials  and $G$-functions are connected to
representation theory and geometry in a way that leads to a
generalization of Schubert calculus where combinatorics is again at
the forefront. Moreover, fundamental aspects of the theory of
Schur functions are contained in the theory of $G$-functions since
each  $G_\lambda$ is an inhomogeneous symmetric polynomial whose
lowest homogeneous component is the Schur function $s_\lambda$.  

Parallel to the study of $G$-functions is the study of a second family
of functions that arise by duality with respect to the Hall inner
product on the ring $\Lambda$ of symmetric functions.  In particular,
results in \cite{Lenart,B02} imply that the $G$-functions form a
Schauder basis for the completion of $\Lambda$ with respect to the
filtration by the ideals $\Lambda^{r} = \bigoplus_{|\lambda|\geq r}
\mathbb Z s_\lambda$.  The dual Hopf algebra to this completion is
isomorphic to $\Lambda$.  Therein lies the basis of $g$-functions,
defined by their duality to the $G$-basis.  Lam and Pylyavskyy first
studied these functions directly in \cite{LP07} where they were called
dual stable Grothendieck polynomials.  By duality, each  $g_\lambda$
is inhomogeneous with highest homogeneous component equal to
$s_\lambda$.  

Strictly speaking, the $G$- and $g$-functions do not lie in
the same space and there is no sensible way to write $G$-functions in
terms of $g$-functions.  However, any element of $\Lambda$ can be 
expanded into both the $G$- and the $g$-functions and it is such
expansions that are of interest here.
Motivated by the many families of symmetric functions whose
transition matrices with Schur functions have combinatorial
descriptions and encode representation theoretic or geometric
information, our focus is on functions with what we refer to as
\emph{tableaux Schur expansions}. A symmetric function, $f_\alpha$, is
said to have a tableaux Schur expansion if there is a set of tableaux
$\bT(\alpha)$ and a weight function $wt_\alpha$ so that 
\begin{equation}
\label{Eq:falpha}
f_\alpha =
\sum_{T \in \bT(\alpha)}^{} wt_\alpha(T) s_{sh(T)}\;.
\end{equation}
Among the classical examples is the family of Hall-Littlewood polynomials 
\cite{Hall,LHall}, whose tableaux Schur expansion gives the 
decomposition of a graded character of $SL_n(\mathbb C)$ into 
its irreducible components \cite{GarPro}.   A more recent 
example is given by the $k$-atoms \cite{[LLM]}.  These  are
conjectured to represent Schubert classes for the homology 
of the affine Grassmannian when $t=1$ and their very
definition is a tableaux Schur expansion.

In this paper, we give combinatorial descriptions for the $G$- and 
the $g$-expansion of every function with a tableaux Schur expansion.
Our formulas are in terms of set-valued tableaux and
reverse plane partitions; $G$-functions are the weight generating 
functions of the former and $g$-functions are the weight generating 
functions of the latter.
More precisely, for any given set $\bT(\alpha)$ of semistandard tableaux,
we describe associated sets $\bS(\alpha)$ and $\bR(\alpha)$ 
of set-valued tableaux and reverse plane partitions, respectively. 
Given also any function $wt_\alpha$ on $\bT(\alpha)$, 
we define an extension of $wt_\alpha$ to $\bS(\alpha)$ and $\bR(\alpha)$.  
In these terms, we prove that any function $f_\alpha$ satisfying
\eqref{Eq:falpha} can be expanded as
\begin{align}
f_\alpha=
   \sum_{S \in \bS(\alpha)}^{} wt_\alpha(S) (-1)^{\varepsilon(S)} g_{sh(S)} 
   = \sum_{R \in \bR(\alpha)}^{} wt_\alpha(R) G_{sh(R)}\,.
\label{eqmain}
\end{align}
The construction of sets $\bS(\alpha)$ and $\bR(\alpha)$ is
described in section~\ref{sec:main}  and the proof of \eqref{eqmain}
is given in section~\ref{sec:involution}.

Since a Schur function has a trivial tableaux Schur expansion, the
simplest application of \eqref{eqmain} describes the transition matrices
between $G/g$ and Schur functions in terms of certain reverse plane 
partitions and set-valued tableaux.  These transition matrices were 
alternatively described by Lenart in \cite{Lenart} using certain skew 
semistandard tableaux.   Although our description is not obviously 
equinumerous, we give a bijective proof of the equivalence in 
section~\ref{sec:altproof}.
As a by-product, we show that Lenart's theorem follows from
\eqref{eqmain}.

In section~\ref{sec:applications}, we show how the description 
of a $G/g$-expansion given by \eqref{eqmain} may lead to a 
more direct combinatorial interpretation for the expansion 
coefficients.  For example, we show that the
Hall-Littlewood functions can be defined by extending the notion of
charge to reverse plane partitions and set-valued tableaux.  We also
show that the $G$/$g$-expansions of a product of Schur functions can
be described by certain Yamanouchi reverse plane partitions and
set-valued tableaux.  Note, this is not the $G$-expansion of 
a product of $G$-functions which was settled in \cite{B02}.   

We use \eqref{eqmain} to get the $G$ and $g$-expansions
for $k$-atoms and Stanley symmetric functions in section
\ref{sec:otherapplications} and leave as open problems their further
simplification.  We finish with a curious identity which has a simple
proof using the methods described here.

\section{Definitions and notation} \label{sec:definitions}
\subsection{Symmetric function basics}
We begin by setting our notation and giving standard definitions
(see eg. \cite{MacSF,EC2,FB} for complete details on symmetric functions).
\begin{Def} \label{Def:YoungDiagram}
  The \emph{Ferrers diagram} of a partition $\lambda = (\lambda_1,
  \lambda_2, \cdots, \lambda_k)$ is a left- and bottom-justified array
  of $1 \times 1$ square cells in the first quadrant of the coordinate
  plane, with $\lambda_i$ cells in the $i^{th}$ row from the bottom.
\end{Def}

\begin{Ex} \label{Ex:YD}
  The Ferrers diagram of the partition $(3,2)$ is $\tiny\yng(2,3)\;$.
\end{Ex}

Given any partition $\lambda$, the \emph{conjugate} $\lambda'$
is the partition obtained by reflecting the diagram of
$\lambda$ about the line $y=x$.  For example,
the conjugate of $(3,2)$  is $(2,2,1)$.

\begin{Def} \label{Def:SST}
  A \emph{semistandard tableau} of shape $\lambda$ is a filling of
  the cells in the Ferrers diagram of $\lambda$ with positive integers,
  such that the entries
  \begin{itemize}
    \item are weakly increasing while moving rightward across any row,
      and 
    \item are strictly increasing while moving up any column.
  \end{itemize}
\end{Def}
\begin{Ex} \label{Ex:SST}
  A semistandard tableau of shape $(3,2)$ is $\tiny\young(23,112)\;$.
\end{Ex}
Throughout this paper, the letter $T$ will generally refer to a
tableau, and $\bT$ will typically denote a set of tableaux.

The \emph{evaluation} of a semistandard tableau is the sequence
$(\alpha_i)_{i \in \naturals}$ where $\alpha_i$ is the number of cells
containing $i$.  The evaluation of the tableau in Example~\ref{Ex:SST} is
$(2,2,1)$ (it is customary to omit trailing $0$'s).  We use
$SST(\lambda)$ to denote the set of all semistandard tableaux of
shape $\lambda$, and $SST(\lambda,\mu)$ to denote the set of all
semistandard tableaux of shape $\lambda$ and evaluation $\mu$.

\begin{Def} \label{Def:ReadingWord}
  A \emph{word} is a finite sequence of positive integers.  The
  \emph{reading word} of a tableau $T$, which we denote by $w(T)$, is
  the sequence $(w_1, w_2, \dots, w_n)$ obtained by listing the
  elements of $T$ starting from the top-left corner, reading
  across each row, and then continuing down the rows.
\end{Def}
\begin{Ex} \label{Ex:RW}
  We have $w\left(\tiny\young(23,112)\right) = (2,3,1,1,2)$.
\end{Ex}

We use the fundamental operations jeu-de-taquin
\cite{S77} and $RSK$-insertion \cite{R38,S61,K70} on words.
The reader can find complete details and definitions of 
these operations in \cite{LSmonoid,EC2,Ful}.
A key property of $RSK$-insertion is that
\[ RSK(w(T)) = T\,, \]
for any tableau $T$. 
When two words insert to the same tableau under the
RSK map, they are said to be \emph{Knuth equivalent}.

The weight generating function of semistandard tableaux can be used
as the definition of \emph{Schur functions}.  For any tableau $T$, let
$\bx^{ev(T)} = x_1^{\alpha_1} x_2^{\alpha_2} \cdots$, where
$(\alpha_1, \alpha_2, \cdots)$ is the evaluation of $T$.  
\begin{Def} \label{Def:Schur}
  The \emph{Schur function} $s_\lambda$ is defined by 
  \[ s_\lambda = \sum_{T \in SST(\lambda)}^{} \bx^{ev(T)} \;.\]
\end{Def}
The Schur functions are elements of $\bZ[ [x_1, x_2, \cdots] ]$, the
power series ring in infinitely many variables, and are well
known to be a basis for the \emph{symmetric functions} (\ie those
elements of $\bZ[ [x_1, x_2, \cdots] ]$ which are invariant
under any permutation of their indices).
\begin{Ex} \label{Ex:S21}
  The Schur function $s_{(2,1)}$ is
  \[ s_{(2,1)} =
  x_1^{2}x_2 + x_1 x_2^2 + 2 x_1 x_2 x_3 + \cdots \] 
  corresponding to the tableaux
  \[\tiny \young(2,11) \quad \young(2,12) \quad \young(3,12) \quad
  \young(2,13) \cdots \;. \]
\end{Ex}

Another basis for the symmetric functions is given by 
the \emph{monomial symmetric functions}.
\begin{Def} \label{Def:Monomial}
  The \emph{monomial symmetric function} $m_\lambda$ is defined by 
  \[ m_\lambda = \sum_{\alpha}^{} \bx^{\alpha} \; ,\] 
  summing over all distinct sequences $\alpha$ which are a
  rearrangement of the parts of $\lambda$. (Here $\lambda$ is thought
  of as having finitely many non-zero parts, followed by infinitely
  many $0$ parts.)
\end{Def}
\begin{Ex} \label{Ex:M21}
  The monomial symmetric function $m_{(2,1)}$ is 
  \[ m_{(2,1)} = x_1^{2}x_2 + x_1 x_2^2 + x_1^2 x_3 +x_1 x_3^2 + x_2^2
     x_3 + x_2 x_3^2 + \dots .  \]
\end{Ex}

The \emph{Kostka numbers} give the change of basis matrix between the
Schur and monomial symmetric functions.  For two partitions
$\lambda,\mu$, we define the number $K_{\lambda,\mu}$ to be the number
of semistandard tableaux of shape $\lambda$ and weight $\mu$.  From
the previous definitions, one can see that a consequence of the
symmetry of the Schur functions is that
\begin{align} \label{eq:stom}
  s_{\lambda} = \sum_{\mu}^{} K_{\lambda,\mu} m_{\mu} \;. 
\end{align}

There is a standard inner product on the vector space of symmetric
functions (known as the \emph{Hall inner product}), defined by setting 
\[ \langle s_\lambda, s_\mu \rangle = 
  \begin{cases} 
    1 \quad \text{ if } \lambda=\mu \\ 0 \quad \text{otherwise.} 
  \end{cases} 
\]

The following proposition is a basic, but very useful, fact of linear
algebra.  
\begin{Prop} \label{Prop:DualBasis}
  If $\left( \left\{ a_\lambda \right\}, \left\{ a^*_\lambda \right\}
  \right)$ and $\left( \left\{ b_\lambda \right\}, \left\{ b^*_\lambda
  \right\} \right)$ are two pairs of dual bases for an inner-product
  space, and 
  \begin{align}\label{eq:dba}
    a_\lambda = \sum_{\mu}^{} M_{\lambda,\mu} b_\mu \;,
  \end{align}
  then 
  \begin{align}\label{eq:dbb}
    b^{*}_\mu = \sum_{\lambda}^{} M_{\lambda,\mu} a^*_\lambda \;.
  \end{align}
\end{Prop}
\begin{proof}
  Pairing both sides of (\ref{eq:dba}) with $b_\mu^*$ gives $\langle
  a_\lambda, b_\mu^* \rangle = M_{\lambda,\mu}$. Similarly, pairing
  both sides of (\ref{eq:dbb}) with $a_\lambda$ gives $\langle
  b_\mu^*, a_\lambda \rangle = M_{\lambda, \mu}$. 
  (\ref{eq:dba}) and (\ref{eq:dbb}) are thus equivalent.
\end{proof}

The set of \emph{complete homogeneous symmetric functions},
$\{h_\lambda\}$, are defined to be the basis that is dual 
to the monomial symmetric
functions.  An immediate consequence of (\ref{eq:stom}) and
Proposition~\ref{Prop:DualBasis} is that 
\begin{align} \label{eq:htos}
  h_\mu = \sum_{\lambda}^{} K_{\lambda,\mu} s_\lambda \;.
\end{align}

\subsection{$G$-functions}
Buch introduced the combinatorial notion of set-valued tableaux 
in \cite{B02} to give a new characterization for $G$-functions and 
to prove an explicit formula for
the structure constants of the Grothendieck ring of a Grassmannian
variety with respect to its basis of Schubert structure sheaves.
It is this definition of $G$-functions we use here.

\begin{Def} \label{Def:SVT}
  A \emph{set-valued tableau} of shape $\lambda$ is a filling of the
  cells in the Ferrers diagram of $\lambda$ with sets of positive
  integers, such that 
  \begin{itemize}
    \item the maximum element in any cell is weakly smaller than the
      minimum element of the cell to its right, and
    \item the maximum element in any cell is strictly smaller than the
      minimum entry of the cell above it.
  \end{itemize}
\end{Def}
Another way to view this definition is by saying that the selection of
a single element from each cell (in any possible way) will always give
a semistandard tableau.
\begin{Ex} \label{Ex:SVT}
  A set-valued tableau of shape $(3,2)$ is 
  \newcommand{\Ca}{1,2} 
  \newcommand{\Cb}{2,3}  
  \newcommand{\Cc}{4,5,6}  
\[ S=
  \Yboxdim{20pt}{\tiny\young(3\Cc,\Ca\Cb 3)}\;. \] 
We have omitted the set braces, `$\{$' and `$\}$', here and throughout for clarity
of exposition.
\end{Ex}

The \emph{evaluation} of a set-valued tableau $S$ is the composition
$\alpha = (\alpha_i)_{i \ge 1}$ where $\alpha_i$ is the total number
of times $i$ appears in $S$.  The evaluation of the
tableau in Example~\ref{Ex:SVT} is $(1,2,3,1,1,1)$.  The collection of
all set-valued tableaux of shape $\lambda$ will be denoted
$SVT(\lambda)$, and the subset of these with evaluation $\alpha$ will
be denoted $SVT(\lambda,\alpha)$. We write $k_{\lambda,\mu}$ for the
number of set-valued tableaux of shape $\lambda$ and evaluation $\mu$.
We will typically denote a set-valued tableau with the letter $S$.
Finally, we define the sign of a set-valued tableau, $\varepsilon(S)$,
to be the number elements minus the number of cells:
\[ \varepsilon(S) = |ev(S)| - |shape(S)| \;. \]

A {\it multicell} will refer to a cell in $S$ that contains more than
one letter.  Note that when $S$ has no multicells we view $S$, as a
usual semistandard tableau. In this case, $|ev(S)|=|shape(S)|$, and
$\varepsilon(S) = 0$.

\begin{Def} \label{Def:Grothendiecks}
  For any partition $\lambda$, the Grothendieck function $G_\lambda$ is
  defined by
  \begin{align*}
    G_\lambda &= \sum_{\mu}^{} (-1)^{|\mu| - |\lambda|}
    k_{\lambda,\mu} m_\mu
    = \sum_{S \in SVT(\lambda)}^{} (-1)^{\varepsilon(S)} x^{ev(S)}
  \end{align*}
\end{Def}

For terms where $|\mu| = |\lambda|$, $k_{\lambda\mu}=K_{\lambda\mu}$
since there are no multicells.  Hence
$G_\lambda$ equals $s_\lambda$ plus higher
degree terms.  Since the $G_\lambda$ are known to be symmetric
functions, they therefore form a basis for the appropriate 
completion of $\Lambda$. 

Applying Proposition~\ref{Prop:DualBasis} to this definition gives
rise to the basis $\{g_\lambda\}$ that is dual to $\{G_\lambda\}$
by way of the system, 
\begin{align} \label{eq:htog2}
  h_\mu &= \sum_{\mu}^{} (-1)^{|\mu| - |\lambda|} k_{\lambda,\mu}
\, g_\lambda 
\,,
\end{align}
over partitions $\mu$.  These $g$-functions $g_\lambda$ 
were studied explicitly by Lam and Pylyavskyy in \cite{LP07} where they
showed that $g_\lambda$ can be described as a certain
weight generating function for \emph{reverse plane partitions}.

\begin{Def} \label{Def:PlanePartitions}
   A \emph{reverse plane partition} of shape $\lambda$ is a filling of
   the cells in the Ferrers diagram of $\lambda$ with positive integers,
   such that the entries are weakly increasing in rows and columns.
\end{Def}
\begin{Ex} \label{Ex:PP}
  A reverse plane partition of shape $(3,2)$ is $\tiny\young(12,112)\;$.
\end{Ex}

Following Lam and Pylyavskyy (and differing from some other
conventions) we define the \emph{evaluation} of a reverse plane
partition to be the composition $\alpha = (\alpha_i)_{i \ge 1}$
where $\alpha_i$ is the total number of columns in which $i$ appears.
The evaluation of the reverse plane partition in
Example~\ref{Ex:PP} is $(2,2)$.  The collection of all reverse plane
partitions of shape $\lambda$ will be denoted $RPP(\lambda)$ and the
subset of these with evaluation $\alpha$ will be denoted
$RPP(\lambda,\alpha)$.  We will typically use the letter $R$ to refer
to a reverse plane partition.

\begin{Thm}[Lam-Pylyavskyy] \label{Thm:LP}
  The polynomials $g_\lambda$ have the expansion
  \begin{align*}
    g_\lambda &= \sum_{R \in RPP(\lambda)}^{} \bx^{ev(R)} \;.
  \end{align*}
\end{Thm}
We note that when $|\mu| = |\lambda|$, the entries must be strictly
increasing up columns; hence $g_\lambda$ is equal to $s_\lambda$ plus
lower degree terms.

\section{General formula for $K$-theoretic expansions}
\label{sec:main}

Our combinatorial formula for the $G$- and the $g$-expansion 
of any function with a tableaux Schur expansion is in terms 
of reverse plane partitions and set-valued tableaux, respectively.
The formula relies on a natural association of these objects
with semistandard tableaux which comes about by a careful choice 
of reading word for set-valued tableaux and reverse plane partitions.

\begin{Def} \label{Def:SVTReading}
  The \emph{reading word} of a set-valued tableau $S$, denoted
  by $w(S)$, is the sequence $(w_1, w_2, \dots, w_n)$ obtained by
  listing the elements of $S$ starting from the top-left corner,
  reading each row according to the following procedure, and then
  continuing down the rows.  In each row, we first ignore the smallest
  element of each cell, and read the remaining elements from right to
  left and from largest to smallest within each cell.  Then we read
  the smallest element of each cell from left to right, and proceed to
  the next row.
\end{Def}
\begin{Ex} \label{Ex:SVTword}
  \newcommand{\Ca}{1,2} 
  \newcommand{\Cb}{2,3}  
  \newcommand{\Cc}{4,5,6}  
  The reading word of  the set-valued tableau $S$ in Example~\ref{Ex:SVT}
  is $w(S)=(6,5,3,4,3,2,1,2,3)$.
\end{Ex}

\begin{Def} \label{Def:PPReading}
  Given a reverse plane partition $R$, circle in each column only the
  bottommost occurrence of each letter.  The \emph{reading word} of
  $R$, which we denote by $w(R)$, is the sequence $(w_1, w_2, \dots,
  w_n)$ obtained by listing the circled elements of $R$ starting from
  the top-left corner, and reading across each row and then continuing
  down the rows.
\end{Def}
\begin{Ex}
The reverse plane partition $R$ in 
Example~\ref{Ex:PP} has $w(R)=(2,1,1,2)$.
\end{Ex}

This given, for a set $\bT(\alpha)$ of semistandard
tableaux, we define sets $\bS(\alpha)$ and $\bR(\alpha)$ of set-valued
tableaux and reverse plane partitions, respectively, by 
\begin{align*}
  S \in \bS(\alpha) &\text{ if and only if } RSK(w(S)) \in \bT(\alpha)
  \,, \text{ and }\\
  R \in \bR(\alpha) &\text{ if and only if } RSK(w(R)) \in \bT(\alpha)
  \,.
\end{align*}
Similarly, we can extend any function $wt_\alpha$ defined on
$\bT(\alpha)$ to $\bS(\alpha)$ and $\bR(\alpha)$ by setting
  \[ wt_\alpha(X) = wt_\alpha(RSK(w(X)))\,, \] 
for any $X$ in $\bS(\alpha)$ or $\bR(\alpha)$.

It is in terms of these definitions that we express
the $G$- and $g$-expansions for functions with a tableaux-Schur
expansion.

\begin{Thm} \label{Thm:main}
  Given $f_\alpha$ with a tableaux Schur expansion,
  \begin{align} \label{eq:fins}
    f_\alpha = \sum_{T \in \bT(\alpha)}^{} wt_\alpha(T) s_{sh(T)} \,,
  \end{align}
  we have 
  \begin{align} 
  \label{eq:finG}
  f_\alpha  &=\sum_{R \in \bR(\alpha)}^{} wt_\alpha(R) G_{sh(R)}\\
\label{eq:fing}
     &= \sum_{S \in \bS(\alpha)}^{} wt_\alpha(S) \,
    (-1)^{\varepsilon(S)}\, g_{sh(S)}\,.
  \end{align}
\end{Thm}
\begin{proof}
Writing the $s_{sh(T)}$ on the right hand side of
equation~\eqref{eq:fins} as the sum over tableaux gives
\begin{align} \label{eq:combf}
  f_\alpha &=
    \sum_{T \in \bT(\alpha)}^{} wt_\alpha(T) \sum_{T'
    \in SST(sh(T))} \bx^{ev(T')}.
\end{align}
Similarly, writing the $g_{sh(S)}$ in (\ref{eq:fing}) as the sum
over reverse plane partitions as given by Theorem~\ref{Thm:LP}, we obtain
\begin{align} \label{eq:combgf}
     \sum_{S \in \bS(\alpha)}^{} wt_\alpha(S) \,
    (-1)^{\varepsilon(S)}\, g_{sh(S)}
   &=
    \sum_{S \in \bS(\alpha)}^{} wt_\alpha(S)
\, (-1)^{\varepsilon(S)}
    \sum_{R \in RPP(sh(S))}^{} \bx^{ev(R)} \;.
\end{align}
Since every semistandard tableau can be viewed as a set-valued tableau
and as a reverse plane partition, every monomial term 
in (\ref{eq:combf}) also appears as a term in \eqref{eq:combgf}.  
Thus to prove that \eqref{eq:combf} equals \eqref{eq:combgf}, it 
suffices to show that the terms in (\ref{eq:combgf}) not occurring
in (\ref{eq:combf}) sum to zero.  

In the same way, writing the $G_{sh(S)}$ in (\ref{eq:finG}) as the sum
over set-valued tableaux according to Definition~\ref{Def:Grothendiecks}, we 
find that
\begin{align} \label{eq:combGf}
\sum_{R \in \bR(\alpha)}^{} wt_\alpha(R) \,G_{sh(R)}
   &=
    \sum_{R \in \bR(\alpha)}^{} wt_\alpha(R)
    \sum_{S \in SVT(sh(R))}^{} (-1)^{\varepsilon(S)} \,\bx^{ev(S)} \;.
\end{align}
Again, every term in (\ref{eq:combf}) appears in (\ref{eq:combGf}) and
it suffices to show that the extra terms in \eqref{eq:combGf}
sum to zero.



From these observations, we can simultaneously prove that
\eqref{eq:combf} equals \eqref{eq:combgf} and \eqref{eq:combGf}
by producing a single sign-reversing and weight-preserving
involution. To be precise, in the next section we introduce a map 
$\iota(S,R)=(S',R')$ and prove that it is an involution on the set 
of pairs of $(S,R)$, where $S$ is a set-valued tableau and $R$ 
is a reverse plane partition of the same shape,
satisfying the properties:
\begin{enumerate}
\item $\iota(S,R)=(S,R)$ if and only if $S$ and $R$ are both
semistandard tableaux,
\item $\varepsilon(S)=\varepsilon(S') \pm 1$ when $S$ is not a
  semistandard tableau, and
\item $(RSK(w(S)), RSK(w(R)))=(RSK(w(S')), RSK(w(R')))$\,.
\qedhere
\end{enumerate} 
\end{proof}

\section{The involution} \label{sec:involution}

We introduce basic operations on set-valued tableaux and reverse 
plane partitions called {\it dilation} and {\it contraction}
and will then define the involution $\iota$ in these terms.
To this end, first setting some notation for set-valued tableaux 
and reverse plane partitions will be helpful.  

Given a set-valued 
tableau $S$, let $row(S)$ be the highest row containing a multicell.  
Let $S_{>i}$ denote the subtableau formed by taking only
rows of $S$ lying strictly higher than row $i$.
For a reverse plane partition $R$,
let $row(R)$ denote the highest row containing an entry 
that lies directly below an equal entry. We use the convention that
when $S$ has no multicell, $row(S)=0$ and when no column of $R$ 
has a repeated entry, $row(R)=0$.

\begin{Def}
Given a set-valued tableau $S$, let $c$ be the rightmost multicell 
in $row(S)$ and define $x=x(S)$ to be the largest entry in $c$. 
The dilation of $S$, $di(S)$, is constructed from $S$ by 
removing $x$ from $c$ and inserting it, via $RSK$, into $S_{>row(S)}$.
\end{Def}

\begin{Ex}
Since $row(S)=2$ and $x(S)=6$, 
\[ \newcommand{\Ca}{1,2} 
      \newcommand{\Cb}{3,4}  
      \newcommand{\Cc}{4,5,6}
      \newcommand{\Cd}{2,3}
      \newcommand{\Ce}{4,5}    
di\left(
      \Yboxdim{19pt}
{\tiny
      \young(7,67,\Cb \Cc 8,1\Ca \Cd 5)}
\right)
      =\Yboxdim{19pt}
{\tiny \young(77,66,\Cb \Ce 8,1\Ca \Cd 5)}
\]
\end{Ex}

\begin{Property}
\label{Prop:dilsvt}
For any set-valued tableau $S$, $di(S)$ is a set-valued tableau.
\end{Property}
\begin{proof}
Let $c$ be the rightmost multicell in row $i$
and let $x=x(S)$.
Rows weakly lower than row $i$ in $S'=di(S)$ form a set-valued 
tableau since $S$ is set-valued to start.  For rows above row $i$,
first note that the cell above $c$ is empty or contains a letter 
strictly greater than $x$.  Thus, the insertion of $x$ into row 
$i+1$ puts $x$ in a cell that is weakly to the left of $c$.  Moreover, 
all entries in row $i$ of $S'$ that are
weakly to the left of cell $c$ are strictly smaller than $x$ since
$c$ is a multicell in $S$.  Thus, in $S'$, $x$ is strictly larger
than all entries in the cell below it.  The claim then follows
from usual properties of RSK insertion. 
\end{proof}

\begin{Property}
\label{Prop:dilweight}
For any set-valued tableaux $S$ and $S' = di(S)$,
\[ RSK(w(S)) = RSK(w(S')) \,.\]
\end{Property}
\begin{proof}
Let $i=row(S)$.  For some word $v$, $w(S)$ can be factored as 
$w(S_{>i}) \cdot x(S) \cdot v$ since $x(S)$ is the first letter 
in the reading word of row $i$.  The definition of dilation 
then gives that the word of $S'$ is 
$w(S_{>i}\from x)\cdot v$.  Thus the Knuth equivalence classes
of $w(S)$ and $w(S')$ are the same since
RSK insertion preserves Knuth equivalence.
\end{proof}

We remark that it is Property~\ref{Prop:dilweight} which 
motivated our definition for the reading word of a set-valued
tableau.

\begin{Def}
Given a reverse plane partition $R$, let $i = row(R)$ 
and let $c$ be the rightmost cell in row $i+1$
that contains the same entry as the cell below it.
The contraction of $R$, $co(R)$, is constructed by replacing
$c$ with a marker and using reverse jeu-de-taquin to slide 
this marker up and to the right until it exits the diagram.  
\end{Def}
\begin{Ex}
\[
co\left({\tiny \young(33,223,112,111)}\right) =
    {\tiny \young(3,233,122,111)} 
\]
\end{Ex}

\begin{Property}
\label{Prop:revweight}
For any reverse plane partitions $R$ and $R' = co(R)$,
\[ RSK(w(R)) = RSK(w(R')) \,.\]
\end{Property}
\begin{proof}
Let $i=row(R)$ and note that the portion of the reading word of $R$
obtained by reading rows weakly below row $i$ is unchanged
by contraction.  Moreover, the rows of $R$ higher than $i+1$
form a semistandard tableau and thus the jeu-de-taquin 
moves in these rows preserve Knuth equivalence.  In row $i+1$,
any initial rightward  slide of the marker does not change 
the reading word since every letter to the right of the 
marker is strictly greater than the letter below it (and 
hence, strictly greater than the letter below it and to its left).  
It thus suffices to check that the move taking 
    the marker from row $i+1$ to row $i+2$  preserves Knuth equivalence.
    To this end, let $\widehat R$ denote the stage of the jeu-de-taquin process at which
    the next move takes the empty marker from row $i+1$ to $i+2$.  Consider
    the subtableau $\widehat R^*$ consisting only of the letters in
    rows $i+1$ and $i+2$ which contribute to the reading word.
    In general, the form of $\widehat R^*$ will be 
    \[ \tiny\young(uxv,w\bullet y) \]
    where $y,v$ are weakly increasing words with $\ell(y) \ge \ell(v)$
    and each $y_i < v_i$, $x$ is a letter with $x \le y_1 < v_1$, and
    $u,w$ are weakly increasing words with $\ell(w) \le \ell(u)$ (since an entry 
    of $w$ may not be the lowest in its column and thus not part of the reading word
whereas all entries of $u$ contribute to reading word since they are in row $i+2$).
    It remains to show the Knuth equivalence of the words
    $uxvwy$ and $uvwxy$.  We first note that the insertion of the word
    $uw$ is 
    \[ \newcommand{\upr}{u'} \newcommand{\uppr}{u''}
    \young(\upr,w\uppr) \]
    for some decomposition of $u$ into disjoint subwords $u', u''$.
    With this observation, it is not hard to verify that both words
    $uxvwy$ and $uvwxy$ insert to 
    \[ \newcommand{\upr}{u'} \newcommand{\uppr}{u''}
       \young(\upr v,w\uppr xy)\]
    and hence these words are Knuth equivalent.  
\end{proof}

\begin{Def}
For a set-valued tableau $S$ and a reverse plane partition $R$
of the same shape as $S$, define the map 
$$
\iota: (S,R) \rightarrow (S',R')
$$
according to the following four cases
(where $y(S,R)$, and the dilation and contraction of a pair
are defined below):
\begin{enumerate}
\item if $row(S)=row(R)=0$, the pair $(S,R)$ is a fixed point
\item if $i=row(S)>row(R)$, the pair is dilated
\item if $i=row(R)>row(S)$, the pair is contracted
\item if $row(R)=row(S)$, the pair is dilated 
when $x(S)\geq y(S,R)$ and is otherwise contracted.
\end{enumerate}
In case (2), $S'=di(S)$ and $R'$ is constructed
from $R$ by replacing the cell of $R$ in position
$(i,j)=S' \setminus S$ with an empty marker and sliding 
this marker to the south-west using jeu-de-taquin.  
When the marker reaches row $i$, we replace it by the entry in the cell
directly above it.  In case (3), $R'=co(R)$. 
Construct $S'$ from $S$ by deleting the cell of $S$ in 
position $(i,j)=R\setminus R'$ and reverse RSK bumping
its entry until the entry $y=y(S,R)$ is bumped from
row $i+1$. Finally, add entry $y$ to the unique cell of 
row $i$ where $y$ is maximal in its cell and the non-decreasing 
row condition is maintained.
Case (4) reduces to case (2) or (3), determined by comparing
the entry $y(S,R)$ to the number $x(S)$ described in the definition of
dilation.
\end{Def}

  \begin{Ex} \label{Ex:map}
    The involution $\iota$ exchanges the two pairs below:
    \[ \newcommand{\Ca}{1,2} 
      \newcommand{\Cb}{3,4}  
      \newcommand{\Cc}{4,5,6}
      \newcommand{\Cd}{2,3}
      \newcommand{\Ce}{4,5}    
      \left( \; \Yboxdim{24pt}
      \young(7,67,\Cb \Cc 8,1\Ca \Cd 5)\normalsize,\young(2,13,123,1134)
      \; \right) 
      \leftrightarrow \left( \; \Yboxdim{18pt}
      \young(77,66,\Cb \Ce 8,1\Ca \Cd 5)\normalsize, 
      \young(23,12,123,1134) \; \right) \]
The pair $(S,R)$ on the left has $row(S)=row(R)=2$ and $x = y = 6$,
implying that $(S,R)$ is dilated under $\iota$.  The pair $(S',R')$ on 
the right has $row(S)=row(R) = 2$ and $x = 5<y=6$ and is thus contracted.  
Note that ${\varepsilon(S)}={\varepsilon(S')+1}$ and that
    \[ (RSK(w(S)),RSK(w(R)))=(RSK(w(S')), RSK(w(R'))) = 
    \left( \; \tiny\young(7,6,57,46,34,238,1125), \young(3,22,1134)
    \right)\,. \]
Thus, $\iota$ reverses the sign and preserves the weight of
this pair.
  \end{Ex}
 
\begin{Prop}
The map $\iota$ is a sign-reversing and weight-preserving 
involution on the set of pairs of $(S,R)$, where $S$ is 
a set-valued tableau and $R$ is a reverse plane partition 
of the same shape.  
The fixed points of $\iota$ are pairs $(S,R)$
where $S$ and $R$ are both semistandard tableau.
\end{Prop}
  \begin{proof}
We first verify that $(S',R')=\iota(S,R)$ is in fact a pair where
$S'$ is a set-valued tableau and $R'$ is a reverse plane partition.
When $\iota$ requires dilation,
Property~\ref{Prop:dilsvt} assures that $S'$ is a valid set-valued tableaux and
it is straightforward to verify that $R'$ is a valid reverse plane partition
by properties of jeu-de-taquin.  

In the case that $\iota$ involves contraction, $R'$ is a reverse plane
partition again by properties of jeu-de-taquin.  Since $i=row(R)\geq row(S)$,
$S_{>i}$ is a semistandard tableau implying by RSK that $S'_{>i}$ is as
well and that $y$ is well-defined.  It remains to check that there is a unique
cell in row $i$ of $S$ into which $y$ can be placed so that it is maximal in
this cell and row $i$ maintains the non-decreasing condition.  We claim that
this cell is the rightmost cell $m$ of row $i$ whose entries are all strictly
less than $y$.  Note that $m$ exists since the cell of $S$ directly below the
cell from which $y$ was bumped has only entries smaller than $y$.  We claim
that $m$ is the unique cell into which $y$ can be placed; namely, no cell to
the right of $m$ contains an element strictly less than $y$.  This is clear
when $i>row(S)$ since then $S$ has no multicells in row $i$.  Otherwise,
conditions of case (4) imply that the largest element of a multicell in row
$i$ is $x(S)<y$.  Hence there are no multicells to the right of $m$ and the
claim follows.  Finally, we observe that the entry in the cell
directly above $m$ (if it exists) must be strictly greater than $y$,
since this cell is weakly to the right of the cell from which $y$ was
bumped.

We now show that $\iota$ is indeed an involution by proving that
if $(S',R')$ is obtained by dilation then $\iota(S',R')$
will require contraction, and vice versa.  Consider the case 
that $\iota$ requires dilation. The definition of $row$ 
implies that $row(R')=i\geq row(S')$ given $i=row(S)\geq row(R)$.
Further, the reversibility of the RSK algorithm and jeu-de-taquin 
on semistandard tableaux give that $y(R',S')=x(S)>x(S')$.
Therefore, applying $\iota$ to $\iota(S,R)$ requires the
contraction case.  The cases in which $\iota$ requires contraction
to start follow similarly.

That $\iota$ is sign-reversing and has the appropriate fixed point 
set is easy to verify from the definition and it remains only to 
show that $\iota$ is weight-preserving. From the definition of $\iota$, and
the fact that $\iota$ is an involution, we have either:
\begin{enumerate}
  \item $R' = co(R)$ and $S = di(S')$, or
  \item $R = co(R')$ and $S' = di(S)$.
\end{enumerate}
In either case, that $\iota$ is weight-preserving follows from
Properties~\ref{Prop:dilweight} and \ref{Prop:revweight}.
\end{proof}

\section{Schur expansions and an alternate proof} \label{sec:altproof}

Lenart proved in \cite{Lenart} that the transition matrices 
between $G$ and Schur functions have a beautiful combinatorial 
interpretation in terms of objects that have since been called 
{\it elegant fillings} in \cite{LP07}.  Here we show how 
the simplest application of Theorem~\ref{Thm:main} gives rise to 
a new interpretation for these transition matrices in terms
of certain reverse plane partitions and set-valued tableaux.

Alternatively, we give a bijection between the elegant fillings 
and these reverse plane partitions/set-valued tableaux.  As a 
by-product, we have an alternate proof for Lenart's result 
following from Theorem~\ref{Thm:main} and vice versa.


\subsection{A new approach to Schur and $G$/$g$-transitions}

\begin{Def}
An elegant filling is a skew semistandard tableaux with the property
that the numbers in row $i$ are no larger than $i-1$.  An elegant
filling whose entries are strictly increasing across rows is called
strict. We let $f_\lambda^{\mu}$ denote the number of elegant
fillings of shape $\lambda / \mu$ and $F_\mu^{\lambda}$ denote the
number of strict elegant fillings of shape $\mu/ \lambda$.  
\end{Def}

\begin{Thm}\cite{Lenart} \label{Thm:Lenart}
  The transition matrices between the Schur functions and the
  $G$-functions are given by the following:
\begin{align}
  \label{lenthmG}
  s_\mu &= \sum_{\lambda}^{} f_{\lambda}^{\mu} G_{\lambda}\\
  G_\lambda &= \sum_{\mu}^{} (-1)^{|\lambda| + |\mu|} 
    F_{\mu}^{\lambda} s_\mu\,.
\end{align}
\end{Thm}

Note that the transition between $g$ and Schur functions follows
immediately by duality:
\begin{align} 
  g_\lambda &= \sum_{\mu}^{} f_{\lambda}^{\mu} s_\mu \\
\label{lenthmg}
  s_\mu &= \sum_{\lambda}^{} (-1)^{|\lambda| + |\mu|} 
    F_{\mu}^\lambda g_\lambda\,.
\end{align}

The simplest application of Theorem~\ref{Thm:main} provides a new
combinatorial description for the $f_\mu^\lambda$ and $F_\lambda^\mu$ 
coefficients.  

\begin{Prop}
\label{Prop:fF}
Fix a partition $\lambda$ and a semistandard tableau $T$ of shape $\mu$. 
$F_\mu^{\lambda}$ is the number of set-valued tableaux of shape $\lambda$ 
whose reading word is equivalent to $w(T)$ and
$f_\lambda^{\mu}$ is the number of reverse plane partitions of 
shape $\lambda$ whose reading word is equivalent to $w(T)$.
\end{Prop}
\begin{proof}
Consider the simple case that $\bT$ consists of just one tableau 
$T$ of shape $\mu$.  We can then 
apply Theorem~\ref{Thm:main} to the trivial expansion
\begin{align} \label{eq:s_in_s}
  s_\mu = \sum_{T \in \bT}^{} s_{sh(T)}
\end{align}
to find that
\begin{align} \label{eq:s_in_g}
  s_\mu\; = \;\sum_{S \in \bS}^{} (-1)^{\varepsilon(S)}\,g_{sh(S)}
\;=\;  \sum_{R \in \bR}^{} G_{sh(R)} \,,
\end{align}
where $\bS$ is the set of all set-valued tableaux whose reading word is Knuth
equivalent to $w(T)$ and $\bR$ is the set of all reverse plane partitions 
whose reading word is Knuth equivalent to $w(T)$.
The result on $F_\mu^\lambda$ then follows by \eqref{lenthmg} and
the interpretation for $f_\lambda^\mu$ follows from \eqref{lenthmG}.
\end{proof}

\subsection{Bijections}

Here we describe bijections between elegant fillings and certain 
reverse plane partitions, and between strict elegant fillings and
certain set-valued tableaux.  The bijections lead to
alternate proofs for Proposition~\ref{Prop:fF}, Theorem~\ref{Thm:Lenart}, 
and Theorem~\ref{Thm:main}.

We first consider a map on strict elegant fillings.  
Recall that $row(S)$ is the highest row of $S$ with a multicell.

\begin{Def}
For any fixed tableau $T$, we define the map
$$
\phi_T:
\left\{
S \in SVT(\mu) : w(S) \sim w(T)
\right\}
\to
\{\text{strict elegant fillings of shape}\; sh(T)/\mu
\}
$$
where $\phi_T(S)$ is the filling of $sh(T)/\mu$ 
that records the sequence of set-valued tableaux 
\[ S = S_0 \to di(S) = S_1 \to di(S_1) = S_2 \to \dots \to S_r = T \]
by putting in cell $sh(S_i) / sh(S_{i-1})$,
the difference between the row index of this cell and $row(S_{i-1})$.
\end{Def}

\begin{Ex} \label{Ex:EF-SVTbij}
  \newcommand{\Ca}{1234} 
  \newcommand{\Cb}{23}  
  \newcommand{\Cc}{123}  
  \newcommand{\Cd}{12}  
Given $\mu=(3,2,1)$ and $T=  {\tiny\young(44,33,222,1111)}$,
  \begin{align*}
    \phi_T\left(\Yboxdim{20pt}\young(4,2\Cb,11\Ca) \right)
=\young(23,\hfil2,\hfil\hfil1,\hfil\hfil\hfil) 
  \end{align*}
is constructed by recording
the sequence of dilations 
  \begin{align*}
    \Yboxdim{20pt}\young(4,2\Cb,11\Ca) &\to&
    \Yboxdim{20pt}\young(4,3,22,11\Ca) &\to&
    \Yboxdim{18pt}\young(4,3,224,11\Cc) &\to&
    \young(4,34,223,11\Cd) &\to&
    \young(44,33,222,1111)\\
    \normalsize
    \young(\hfil,\hfil\hfil,\hfil\hfil\hfil) &&
    \young(2,\hfil,\hfil\hfil,\hfil\hfil\hfil) &&
    \young(2,\hfil,\hfil\hfil1,\hfil\hfil\hfil) &&
    \young(2,\hfil2,\hfil\hfil1,\hfil\hfil\hfil) &&
    \young(23,\hfil2,\hfil\hfil1,\hfil\hfil\hfil) 
  \end{align*}
\end{Ex}

\bigskip

\begin{Prop}
\label{phi}
For any tableau $T$, $\phi_T$ is a bijection.
\end{Prop}
\begin{proof}
Fix a tableau $T$ and let $\lambda$ denote its shape.
Consider a set-valued tableau $S$ of shape $\mu$ whose reading 
word is equivalent to $w(T)$.  

We start by showing that $F=\phi_T(S)$ is a 
strict elegant filling of shape $\lambda / \mu$.  Let
\[ S = S_0 \to di(S) = S_1 \to di(S_1) = S_2 \to \dots \to S_r = T. \]
Note that this procedure indeed ends with $T$ since each dilation preserves
the Knuth equivalence class of the reading word.  
Since each $sh(S_i)\subset sh(S_{i+1})$, and $row(S_{i})\geq row(S_{i+1})$,
$F$ is an elegant filling of the correct shape by construction.  
To ensure that $F$ is a {\it strict} elegant filling, it is enough to 
know that the bumping paths created by successive dilations starting 
in the same row do not terminate in the same row. 
Since dilation starts with the largest entry in a row, 
such successive dilations involve bumping from row $r$, 
a letter $z$ after an $x$ where $z\leq x$.  Therefore,
the bumping path created by $x$ must be weakly inside the
bumping path of $z$ and in particular, terminates in a higher 
row.

It remains to show that $\phi_T$ is invertible.  
For the inverse map, consider a strict elegant filling of shape $\lambda /
\mu$. In the elegant filling, we first
replace each entry $i$ by $r-i$ where $r$ is the row index of the cell
containing $i$.  The resulting filling consists of ``destination rows'' for the
corresponding entries in our fixed tableau $T$. We proceed by performing 
contraction on the
entries of $T$ which are outside of the inner shape $\mu$, stopping the
reverse-bumping procedure when we get to the destination row.  The order
in which
these contractions are performed is determined first by the destination rows
(smallest to largest) and then by the height of the original cell (highest to
lowest). This concludes the proof.
\end{proof}

Recall that $row(R)$ is the row index of the highest cell in a 
reverse plane partition $R$ which contains 
the same entry as the cell immediately above it.

\begin{Def}
For any fixed tableau $T$, we define the map
$$
\psi_T:
\left\{
R \in RPP(\lambda) : w(R) \sim w(T)
\right\}
\to
\{\text{elegant fillings of shape}\; \lambda/sh(T)
\}
$$
where $\psi_T(R)$ is the filling of $\lambda/sh(T)$ 
that records the sequence of reverse plane partitions
\[ R = R_0 \to co(R) = R_1 \to co(R_1) = R_2 \to \dots \to R_r = T \]
by putting $row(R_i)$ in cell $sh(R_i) / sh(R_{i+1})$.
\end{Def}

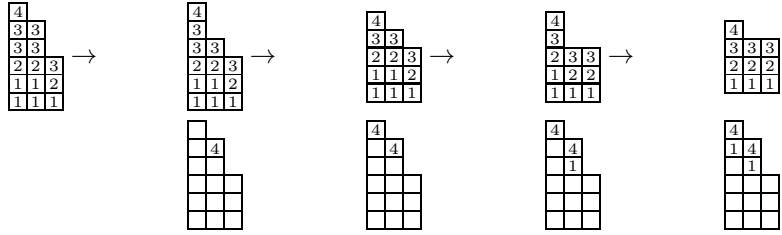
\begin{Ex} \label{Ex:EF-RPPbij}
  \begin{align*}
\tiny
    \young(4,33,33,223,112,111) &\to&
\tiny
    \young(4,3,33,223,112,111) &\to&
\tiny
    \young(4,33,223,112,111) &\to&
\tiny
    \young(4,3,233,122,111) &\to&
\tiny
    \young(4,333,222,111) \\
    && 
\tiny
\young(\hfil,\hfil4,\hfil\hfil,\hfil\hfil\hfil,\hfil\hfil\hfil,\hfil\hfil\hfil)
    &&\tiny \young(4,\hfil4,\hfil\hfil,\hfil\hfil\hfil,\hfil\hfil\hfil,\hfil\hfil\hfil)
    &&\tiny \young(4,\hfil4,\hfil1,\hfil\hfil\hfil,\hfil\hfil\hfil,\hfil\hfil\hfil)
    && \tiny\young(4,14,\hfil1,\hfil\hfil\hfil,\hfil\hfil\hfil,\hfil\hfil\hfil)
  \end{align*}
\end{Ex}

\begin{Prop}
\label{psi}
For any semistandard tableau $T$, $\psi_T$ is a bijection.
\end{Prop}
\begin{proof}
Fix a tableau $T$ and denote its shape by $\mu$.
Consider a reverse plane partition $R$ of shape $\lambda$ 
whose reading word is equivalent to that of $T$ and set $F=\psi_T(R)$.
Let
\[ R = R_0 \to co(R) = R_1 \to co(R_1) = R_2 \to \dots \to R_r = T \]
and note that this process does terminate in $T$ since each contraction 
preserves the equivalence class of the reading word.  
By construction, $F$ has weakly increasing rows, is of the proper shape,
and has only entries less than $i$ in row $i$.  To see that the entries of $F$
are strictly increasing up columns, it is enough to note that the reverse
jeu-de-taquin paths of two cells from the same row cannot intersect.  Hence
$F$ is an elegant filling.  To invert this procedure, we repeatedly perform
dilation on $T$, by placing an empty marker in a cell of $\lambda / \mu$, and
sliding this marker into $T$, terminating at the row indicated by the
corresponding entry of $F$. The order these contractions are performed is
determined first by the entries of $F$ (ordered smallest to largest) breaking
ties by proceeding from left to right.
\end{proof}

\begin{Cor}
Theorem~\ref{Thm:Lenart} follows from Theorem~\ref{Thm:main}.
\end{Cor}
\begin{proof}
We apply Theorem~\ref{Thm:main} to the trivial expansion
\eqref{eq:s_in_s} to find that
\begin{align} 
  s_\mu\; = \;\sum_{S \in \bS}^{} (-1)^{\varepsilon(S)}\,g_{sh(S)}
\;=\;  \sum_{R \in \bR}^{} G_{sh(R)} \,,
\end{align}
where $\bS$ is the set of all set-valued tableaux whose reading word is Knuth
equivalent to $w(T)$ and $\bR$ is the set of all reverse plane partitions 
whose reading word is Knuth equivalent to $w(T)$.
The results then follow from Propositions~\ref{phi} and \ref{psi}.
\end{proof}
\begin{Cor}
Theorem~\ref{Thm:main} follows from
Theorem~\ref{Thm:Lenart}.
\end{Cor}
\begin{proof}
Given Theorem~\ref{Thm:Lenart}, we reinterpret the coefficients
using Propositions~\ref{phi} and \ref{psi} to obtain
\eqref{eq:s_in_g}.  Our claim follows by linear extension.
\end{proof}

\section{Applications} \label{sec:applications}

Here we give a direct combinatorial characterization for the
$G$- and $g$-expansions of products of Schur functions and 
certain Macdonald polynomials.  
Theorem~\ref{Thm:main} can be applied to any function $f$ with a 
tableaux Schur expansion to obtain the $G$- and $g$-expansion for $f$.
The expansion coefficients are given by the enumeration of
reverse plane partitions (resp. set-valued tableaux) whose reading word 
is Knuth equivalent to prescribed sets of tableaux.  We use such 
descriptions as a springboard for finding combinatorial interpretations 
for the coefficients that avoid Knuth equivalence.

\subsection{Littlewood-Richardson expansions}
Perhaps the most classical example of a tableaux Schur expansion is 
the Littlewood-Richardson rule for multiplying Schur functions.
The coefficients in
\begin{equation} \label{litric}
  s_\mu\, s_\nu = \sum_{\lambda}
  c_{\mu\nu}^\lambda s_\lambda 
\end{equation}
have beautiful combinatorial descriptions,
count multiplicities of the irreducible $GL(n,\mathbb C)$-module
$V^\lambda$ of highest weight $\lambda$ in the tensor product of
$V^\mu \otimes V^\nu$, and encode the number of points in the
intersection of certain Schubert varieties of the Grassmannian.

To describe the numbers $c_{\mu\nu}^\lambda$ requires that a 
certain {\it Yamanouchi} condition be described on words. 
A word $w = w_1, \dots, w_j$ satisfies the Yamanouchi
condition with respect to the letters $a_1 < a_2 < \dots < a_\ell$ when,
in every rightmost segment of $w$, $w_i w_{i+1} \dots w_j$, the number
of $a_r$'s is greater than or equal to the number of $a_{r+1}$'s for
all $1 \le r < \ell$.  For example, the following words are examples of
Yamanouchi words with respect to the letters $\{3, 4, 5\}$:
\[ 54433 \qquad 45343 \qquad 15432326. \]

The numbers $c_{\mu\nu}^\lambda$ can now be described as the number of
skew semistandard Young tableaux of shape $\lambda / \mu$ which have
evaluation $\nu$ and whose reading word satisfies the Yamanouchi
condition.  This was first stated by Littlewood-Richardson \cite{LR34}
and first proved by Sch\"utzenberger~\cite{S77} and
Thomas~\cite{T74,T77}. Note that Theorem~\ref{Thm:main} cannot be
applied directly to (\ref{litric}), since $c_{\mu\nu}^\lambda$ counts
tableaux of shape $\lambda / \mu$, not shape $\lambda$. However, in
this case, this is easily remedied.

We begin by defining $\bT(\mu, \nu)$ as the set of tableaux
with evaluation $(\mu_1, \dots, \mu_\ell,$ $\nu_1, \dots,
\nu_k)$ whose reading words satisfy the Yamanouchi 
condition with respect to the letters $1, \dots, |\mu|$ and
(separately) with respect to the letters $|\mu| + 1, \dots, |\mu| +
|\nu|$.

\begin{Lem} \label{Lem:LRRule}
  We have the tableaux Schur expansion
  \begin{equation}
    s_\mu s_\nu = \sum_{T \in \bT(\mu,\nu)}^{} s_{sh(T)} \,.
  \end{equation}
\end{Lem}
\begin{proof}
  It is not hard to see that any tableau whose reading word satisfies the
  Yamanouchi condition with respect to the letters $1, \dots, |\mu|$ must
  contain as a subtableau the unique tableau of shape and evaluation
  $\mu$.  Hence there is a bijection between the tableaux counted by
  $c_{\mu\nu}^{\lambda}$ and those tableaux in $\bT(\mu,\nu)$ of shape
  $\lambda$, given by adding $|\mu|$ to every letter and filling in
  the inner shape $\mu$ with the unique tableau of shape and evaluation
  $\mu$.
\end{proof}

We can apply Theorem~\ref{Thm:main} to this tableaux Schur expansion 
to get the $G$ and $g$-expansion of $s_\mu s_\nu$ as a sum over
elements of $\bS(\mu,\nu)$ and $\bR(\mu,\nu)$.  In fact, the 
elements of $\bS(\mu,\nu)$ and $\bR(\mu,\nu)$ can be more directly
characterized in terms of the Yamanouchi condition.  

\begin{Lem}
Consider $\bT(\mu,\nu)$ as defined above.
A set-valued tableau $S\in \bS(\mu, \nu)$
(resp. $R\in \bR(\mu,\nu)$) if and only if
$S$ (resp. $R$) has evaluation $(\mu_1, \dots, \mu_\ell, \nu_1, \dots,
\nu_k)$ and its reading word satisfies the Yamanouchi 
condition with respect to the letters $1, \dots, |\mu|$ and
(separately) with respect to the letters $|\mu| + 1, \dots, |\mu| +
|\nu|$.  
\end{Lem}
\begin{proof}
By definition, the reading word of $S\in\bS(\mu,\nu)$ (resp.
$R\in\bR(\mu,\nu)$) is Knuth equivalent to an  element of
$\bT(\mu,\nu)$.  The result then follows from the fundamental property
that the Yamanouchi property is preserved under Knuth
equivalence \cite{Ful}.  
\end{proof}

\begin{Cor}
\label{Cor:LRRule}
For partitions $\mu = (\mu_1, \dots, \mu_\ell)$ and
$\nu = (\nu_1, \dots, \nu_k)$, 
 \begin{align*}
  s_\mu s_\nu  &= \sum_{S \in S(\mu, \nu)}^{} (-1)^{\varepsilon(S)}
  g_{sh(S)} \\
  &= \sum_{R \in R(\mu, \nu)}^{} G_{sh(R)}\,.
\end{align*}
where $S(\mu, \nu)$ (resp. $R(\mu,\nu)$) is the set of 
set-valued tableaux (reverse plane partitions)
with evaluation $(\mu_1, \dots, \mu_\ell, \nu_1, \dots,
\nu_k)$ whose reading words satisfy the Yamanouchi 
condition with respect to the letters $1, \dots, |\mu|$ and
(separately) with respect to the letters $|\mu| + 1, \dots, |\mu| +
|\nu|$.  
\end{Cor}

\subsection{Hall-Littlewood symmetric functions}

An acclaimed family of functions with a tableaux Schur expansion 
is the Hall-Littlewood basis $\{H_\lambda[X;t]\}_\lambda$. These 
are a basis for $\Lambda$ over the polynomial ring $\bZ[t]$ 
that reduces to the homogeneous basis when the parameter $t$ is set 
to 1. These often are denoted by $\{Q'_\lambda[X;t]\}$ in the
literature (\cite{MacSF}). Hall-Littlewood polynomials arise and can
be defined in various contexts such as the Hall Algebra, the character
theory of finite linear groups, projective and modular representations
of symmetric groups, and algebraic geometry.  We define them here via
a tableaux Schur expansion due to Lascoux and Sch\"utzenberger
\cite{LS78}.

The key notion is the \emph{charge} statistic on semistandard
tableaux. This is given by defining charge on words and then defining
the charge of a tableau to be the charge of its reading word.  For our
purposes, it is sufficient to define charge only on words whose
evaluation is a partition.  
We begin by defining the charge of a word with weight $(1,1,\dots,1)$,
or a \emph{permutation}.  If $w$ is a permutation of length $n$, then
the charge of $w$ is given by $\sum_{i=1}^{n} c_i(w)$ where $c_1(w) =
0$ and $c_i(w)$ is defined recursively as
\begin{align*}
  c_i(w) &= c_{i-1}(w) + \chi\left( \text{$i$ appears to the right of
  $i-1$ in $w$} \right).
\end{align*}
Here we have used the notation that when $P$ is a proposition, $\chi(P)$ is
equal to $1$ if $P$ is true and $0$ if $P$ is false.
\begin{Ex} \label{Ex:ch1}
  $ch(3,5,1,4,2) = 0 + 1 + 1 + 2 + 2 = 6.$
\end{Ex}

We will now describe the decomposition of a word with partition
evaluation into \emph{charge subwords}, each of which are permutations.
The charge of a word will then be defined as the sum of the charge of
its charge subwords.  To find the first charge subword $w^{(1)}$ of a
word $w$, we begin at the \emph{right} of $w$ (\ie at the last element
of $w$) and move leftward through the word, marking the first $1$ that
we see.  After marking a $1$, we continue to travel to the left, now
marking the first $2$ that we see.  If we reach the beginning of the
word, we loop back to the end.  We continue in this manner, marking
successively larger elements, until we have marked the largest letter
in $w$, at which point we stop.  The subword of $w$ consisting of the
marked elements (with relative order preserved) is the first charge
subword.  We then remove the marked elements from $w$ to obtain a word
$w'$.  The process continues iteratively, with the second charge
subword being the first charge subword of $w'$, and so on.
\begin{Ex} \label{Ex:ch2}
Given $w =(5,2,3,4,4,1,1,1,2,2,3)$, the first charge subword of $w$ 
are the bold elements in
  $(\mathbf{5}, \mathbf{2}, 3, 4, \mathbf{4}, 1, 1, \mathbf{1}, 2, 2,
  \mathbf{3})$.  If we remove the bold letters, the second
  charge subword is given by the bold elements in
$(\mathbf{3}, \mathbf{4}, 1, \mathbf{1}, 2, \mathbf{2})$.  
It is now easy to see that the third and final
  charge subword is $(\mathbf{1}, \mathbf{2})$.  Thus we have the
  following computation of the charge of $w$:
  \begin{align*}
    ch(w) &= ch(5,2,4,1,3) + ch(3,4,1,2) + ch(1,2) \\
    &= (0 + 0 + 1 + 1 + 1) + (0 + 1 + 1 + 2) + (0 + 1) \\
    &= 8
  \end{align*}
Since $w$ is the reading word of the tableau
    \[T= \tiny\young(5,2344,111223)\,,\]
we find that the $ch(T)=8$.
\end{Ex}

An important property of the charge function is that it is compatible
with RSK:
\begin{Property} \cite{LS78} \label{Prop:RSKcharge}
  For all words $w$, $ch(w) = ch(RSK(w))$.
\end{Property}

\begin{Def} \label{Def:HL}
  The Hall-Littlewood polynomial $H_\mu[X;t]$ is defined by
  \begin{align} \label{eq:Htos}
    H_\mu[X;t] =\sum_{T \in \bT(\mu)}^{} t^{ch(T)} s_{sh(T)}
  \end{align}
  where $\bT(\mu)$ is the set of all tableaux of evaluation
  $\mu$.
\end{Def}
Note that when $t=1$, Definition~\ref{Def:HL} reduces to
equation~(\ref{eq:htos}) implying that $H_\mu[X;1]=h_\mu$.

In the same spirit that the charge of a semi-standard tableau
is given by the charge of its reading word, we define the charge 
of a set-valued tableau and of a reverse plane partition by the 
taking the charge of their reading words.  
\begin{Ex}
Since the words of the set valued tableau
  \newcommand{\Ca}{3,4} 
  \newcommand{\Cb}{2,3}  
  \newcommand{\Cc}{4,5}  
  \[ 
  \Yboxdim{16pt}\young(23\Cc,11122\Ca)\normalsize \;. \] 
and the reverse plane partition
\[ \small\young(53,23,1344,111223) \]
are both given by $w =(5,2,3,4,4,1,1,1,2,2,3)$, the charge of both of these
objects is $8$.
\end{Ex}

\begin{Cor} \label{Cor:hallinG}
The Hall-Littlewood functions can be written in terms of $G$-functions as
  \begin{align} \label{eq:thm}
    H_\mu[X;t] &= 
\sum_{\lambda}\sum_{R\in RPP(\lambda,\mu)}
t^{ch(R)}\, G_{\lambda}\;,
  \end{align}
and in terms of $g$-functions as
  \begin{align} \label{eq:thm}
     H_\mu[X;t] &= \sum_\lambda
(-1)^{|\mu| - |\lambda|} 
\sum_{S\in SVT(\lambda,\mu)}
     t^{ch(S)}\, g_{\lambda} \;.
 \end{align}
\end{Cor}
\begin{proof}
The tableaux Schur expansion
\eqref{eq:Htos} for Hall-Littlewood polynomials becomes
\begin{align}
H_\mu[X;t] 
&=
\sum_{R \in \bR(\mu)}^{} t^{ch(RSK(w(R))} G_{sh(R)}\\
&=
    \sum_{S \in \bS(\mu)}^{} t^{ch(RSK(w(S))} 
    (-1)^{\varepsilon(S)} g_{sh(S)}
\end{align}
by Theorem~\ref{Thm:main}. By Property~\ref{Prop:RSKcharge}, we have
that $ch(RSK(w)) = ch(w)$.  Moreover, the set $\bR(\mu)$ (resp.
$\bS(\mu)$) is none other than reverse plane partitions (resp.
set-valued tableaux) of evaluation $\mu$. 
\end{proof}

\subsection{Macdonald polynomials}
The generalization of the Hall-Littlewood polynomials to a two
parameter family of symmetric functions gives a hotly studied 
problem in the area of tableaux Schur expansions.  Ideas impacting
theories ranging from Hilbert schemes of points in the plane to
quantum cohomology 
have come forth
from the study of the Schur expansion coefficients
$K_{\lambda\mu}(q,t)$ in Macdonald polynomials: 
\begin{equation} \label{macex}
  H_\mu[X;q,t] = \sum_\lambda K_{\lambda\mu}(q,t) s_\lambda \,.
\end{equation}
Major progress in the combinatorial study of these coefficients 
during the last decade establishes that
\begin{equation}
   H_\mu[X;q,t] = \sum_{T \in \bT(1^n)} wt_\mu(T)\, s_{sh(T)}
\,,
\end{equation}
where $\bT(1^n)$ is the set of standard tableaux and $wt_\mu$ is 
an unknown statistic associating some $t$ and $q$ power to each 
standard tableaux~\cite{H01, A07}.  

The ongoing search for $wt_\mu$ has led to many exciting theories 
and a natural solution has been found in special cases.  When $q=0$, 
$H_\mu[X;0,t]$ are the Hall-Littlewood polynomials, whose 
expansion was just described in the previous section in terms of
charge.  In addition to the Hall-Littlewood case, there is a
tableaux Schur expansion when $q=1$.  Again, we not only apply our 
theorem to obtain the $G$- and $g$-expansions, 
but we can go further to provide a direct combinatorial
interpretation for the expansion coefficients.

The formula for expanding $H_\lambda[X;1,t]$ into Schur functions 
was given by Macdonald \cite{MacSF} in terms of a refined 
charge statistic.  In particular, for any standard tableau $T$,
he defines
$$
  ch_\mu(T)=\sum_{i \geq 1} ch(\rho_i T)\,,
$$
where $\rho_i T$ is the restriction of $T$ to the $ith$ segment of $[1,n]$
of length $\mu_i'$ (with each entry decremented by $\sum_{j<i}\mu_j'$). 
\begin{Ex}
Let $\mu=(3,3,2,1)$ and
    \[T= \young(9,4568,1237)\,.\]
Note that $\mu'=(4,3,2)$ and thus $ch(\rho_1T)=ch(4123)=5$,
$ch(\rho_2T)=ch(567)=ch(123)=3$,
$ch(\rho_3T)=ch(98)=ch(21)=0$. Hence 
$ch_{(3,3,2,1)}(T) = 5+3+0 = 8$.
\end{Ex}

Macdonald then proves that 
   \[H_\mu[X;1,t] = \sum_\lambda
\sum_{T \in SST(\lambda,1^n)}
t^{ch_\mu(T)}\, s_\lambda \;.\]

Extending the definition of $ch_\mu$ to set-valued tableaux and 
reverse plane partitions in the same way (by restricting and decrementing
the entries of the reading word), Theorem~\ref{Thm:main} gives the
$G$- and $g$-expansions of the $H_\mu[X;1,t]$.
\begin{Cor} \label{Cor:Macq=1}
  The Macdonald polynomials when $q=1$ satisfy the expansions
  \begin{align} \label{eq:thm}
     H_\mu[X;1,t] &= \sum_{\lambda}^{} (-1)^{|\mu| - |\lambda|}
     \sum_{S \in SVT(\lambda,1^n)} t^{ch_\mu(S)} \,g_\lambda \;\\
      &= \sum_{\lambda}^{} 
     \sum_{R \in RPP(\lambda,1^n)} t^{ch_\mu(R)} \,G_\lambda \;.
  \end{align}
\end{Cor}

\section{Other applications and future work} \label{sec:otherapplications}

Here we give the $G$- and $g$-expansion for $k$-atoms and Stanley
symmetric functions where the expansion coefficients count
reverse plane partitions (resp. set-valued tableaux) whose reading word 
is Knuth equivalent to prescribed sets of tableaux.  We leave as
open problems the task of describing the coefficients
without using Knuth equivalence. 

\subsection{k-Atoms}

One study of the tableaux Schur expansion of Macdonald polynomials
led to the discovery of a family of polynomials called atoms.
To be precise, the \cite{[LLM]} study restricted attention to the subspaces 
of Macdonald polynomials $\Lambda^k_t=span \{H_\lambda[X;q,t]\}_{\lambda_1\leq k}$,
for fixed integer $k>0$.  There, for each partition with no part larger than $k$,
a $k$-atom was introduced and defined as
\begin{equation}
s_{\mu}^{(k)}[X;t] = \sum_{T\in \mathcal A^k_\mu}
t^{ch(T)} \, s_{sh(T)}\, \,,
\label{(1.6)}
\end{equation}
for certain prescribed sets of tableaux $\mathcal A^k_\mu$
(see \cite{[LLM]}).
It was conjectured that any Macdonald polynomial 
in $\Lambda^k_t$ can be decomposed as:
\begin{equation}
H_{\lambda}[X;q,t\, ] = \sum_{\mu\in\mathcal P^k}
K_{\mu \lambda}^{(k)}(q,t) \, s_{\mu}^{(k)}[X;t\, ] \quad\text{where}\quad
K_{\mu \lambda}^{(k)}(q,t) \in \mathbb N[q,t] \, .
\label{mackkostka}
\end{equation}
Moreover, it was conjectured that the element $s_\mu^{(k)}[X;t]$
reduces simply to $s_\mu$ for $k\geq |\lambda|$, thus refining
\eqref{macex} since the expansion coefficients
in \eqref{mackkostka} are $K_{\lambda\mu}(q,t)$ for large $k$.


The $k$-atoms are a perfect candidate for Theorem~\ref{Thm:main} since
they have a tableaux Schur expansion by definition.
\begin{Cor} \label{Cor:EG}
For $\bT(\mu,k)=\mathcal A^k_\mu$, we have that
  \begin{align*}
    s_\mu^{(k)}[X;t] &= 
    \sum_{S \in \bS(\mu,k)}^{} (-1)^{\varepsilon(S)} \, t^{ch(S)} g_{sh(S)}\\
    &= \sum_{R \in \bR(\mu,k)}^{} t^{ch(R)} G_{sh(R)} \;.
  \end{align*}
\end{Cor}

\subsection{Stanley symmetric functions}
Let $s_1, \dots, s_{n-1}$ be the simple transpositions which generate the
symmetric group $S_n$.  To each reduced decomposition $\sigma = s_{i_1} \dots
s_{i_n}$, we associate the word $(i_1, \dots, i_n)$. For $\sigma \in S_n$, let
$R(\sigma)$ denote the set of words for reduced decompositions of $\sigma$.
Stanley~\cite{S84} defined symmetric functions
\[ F_\sigma = \sum_{w_1 w_2\cdots w_k \in R(\sigma)}^{} 
\sum_{\substack{i_1 \le i_2 \le \cdots \le i_k\\
                w_j > w_{j+1} \implies i_j < i_{j+1}}}^{} 
     x_{i_1} \dots x_{i_k}. \]
Stanley showed that $F_\sigma$ is always a symmetric function, and conjectured
that it is always Schur-positive.  This conjecture was later proven by
Edelman-Greene~\cite{EG87}, and independently by
Lascoux-Sch\"utzenberger~\cite{LS82a}. Edelman-Greene gave a tableaux
Schur expansion for the Stanley symmetric functions. 

\begin{Thm}[\cite{EG87}] \label{Thm:EG}
  \[ F_\sigma = \sum_{T \in \bT(\sigma)}^{} s_{sh(T)} \,\]
over the set $\bT(\sigma)$ of all tableaux whose reading word is a 
reduced word for $\sigma^{-1}$.
\end{Thm}
Since we have a tableaux Schur expansion, we can apply Theorem~\ref{Thm:main}.
\begin{Cor} \label{Cor:EG}
\[
    F_\sigma = \sum_{S \in \bS(\sigma)}^{} (-1)^{\varepsilon(S)}g_{sh(S)}\,,
\]
where $\bS(\sigma)$ is the set of all set-valued tableaux 
with a reading word which is Knuth equivalent to a reduced word for $\sigma^{-1}$,
and
  \[
F_\sigma = \sum_{R \in \bR(\sigma)}^{} G_{sh(R)} \;,
\]
where $\bR(\sigma)$ is the set of all reverse plane partitions
with a reading word which is Knuth equivalent to a reduced word for $\sigma^{-1}$.
\end{Cor}

It is worth noting that we {\it cannot} describe $\bS(\sigma)$ as
the set of all set-valued
tableaux whose reading word is a reduced word for $\sigma^{-1}$.
For example, consider
  \newcommand{\Ca}{1,2} 
  \newcommand{\Cb}{2,3}  
\[ S = \Yboxdim{18pt}\young(3,\Ca\Cb)\normalsize \;. \] 
The reading word $w(S) = 33212$ is clearly not the reduced word of any
permutation, but $S$ is in $\bS(s_2 s_1 s_3 s_2 s_3)$, since
$w(RSK(w(S))) = 32312$. On the other hand
  \newcommand{\Cd}{2,3,4}  
\[ S' = \Yboxdim{24pt}\young(3,1\Cd)\normalsize \] 
has reading word $34312$ which is a reduced word for the inverse of
the permutation $(s_2 s_1 s_3 s_4 s_3$).  However, the reading word of
$RSK(w(S'))$ is $43312$, so $S'$ is not in $\bS(\sigma)$ for any
$\sigma$.

A similar phenomenon occurs with reverse plane partitions. For example,
with 
\[ T = \young(22,12,13)\,, \]
$w(T)= 2213$ which is not a reduced word.  However, applying
RSK gives a tableau whose reading word is the reduced word $2123$.
On the other hand, given instead
\[ T' = \young(34,13,12)\,, \]
we have that $w(T')=34312$ is reduced, but the insertion gives
a tableau whose reading word is $43312$.

Although Corollary~\ref{Cor:EG} gives an interpretation for the $g$- and
$G$-expansions of a Stanley symmetric function, a characterization 
of the set-valued tableaux and reverse plane partitions appearing 
in these expansions that does not involve RSK-insertion remains an interesting
open problem.

\subsection{Related work}

In \cite{Lpm1}, Lascoux studies transformations on symmetric functions 
under the shift of power sums $p_i \to p_i \pm 1$.  The application
of such a transformation to the Hall-Littlewood polynomials leads to
an expansion for transformed Hall-Littlewood polynomials in terms of 
plane partitions.  Finding the precise connection between this
expansion and the expansion of a Hall-Littlewood polynomial
$H_\mu[X;t]$ given in Corollary~\ref{Cor:hallinG} is an interesting
problem.

The Schubert representatives for the K-theory of affine Grassmannians 
and their dual in the nil Hecke ring are given by a refinement of the 
$G$ and the $g$-functions to families, $G^k$ and $g^k$ \cite{LSS,MoK}. 
These too are inhomogenous functions, but they are indexed by an extra integer 
parameter $k>0$ and now, the dual $k$-Schur 
functions \cite{[LMhecke]} and $k$-Schur functions \cite{[LMproofs]}  
are the lowest and highest degree terms, respectively.  
There is a tableaux definition for $G^k$ in the spirit of 
Definition~\ref{Def:Schur} for Schur functions \cite{MoK}
giving rise to many problems regarding the tableaux combinatorics
of the $G^k/g^k$ families.  One direction along these lines would be 
to find the $G^k/g^k$-expansions of functions with a tableaux-$k$-Schur 
or tableaux-dual-$k$-Schur expansion.

\section{Identities} \label{sec:identities}

\begin{Prop}
Given any function $f_\alpha$ where
\begin{equation}
f_\alpha= \sum_{T \in \bT(\alpha)}^{} wt(T) s_{sh(T)}\,,
\end{equation}
we have that
\begin{equation}
  \sum_{S \in \bS(\alpha)}^{} (-1)^{\varepsilon(S)}\,wt(S)=wt(S_0)\,,
\end{equation}
where $S_0$ is the tableau of row shape.
\end{Prop}
\begin{proof}
We do this by constructing a sign-reversing, weight-preserving involution 
on $\bS(\alpha)$ where $S_0$ is the only fixed point.  The appropriate
candidate is defined as follows; let $n$ denote the largest letter in $S$
and set $x=n$.  Let $row(S)$ be the highest row of $S$ containing $x$
and determine which of the following cases occurs:
\begin{enumerate}
\item there is an $x$ in a multi-element cell of row $row(S)$
\item $row(S)\neq 1$ and no multi-element cells in row $row(S)$ or
$row(S)-1$ contain $x$
\item $row(S)=1$ and there is no $x$ in a multi-element cell of row 1.
\end{enumerate}
When $S$ satisfies condition (1), construct $S'$ 
from $S$ by moving the $x$ in this multi-element
cell to the end of $row(S)+1$.  When we have condition (2), 
construct $S'$ from $S$ by moving the $x$ from the end
of row $row(S)$ into the rightmost cell of $row(S)-1$ containing
only letters strictly smaller than $x$.  In the last case (3), we 
replace $x$ by $x-1$ and determine which case is satisfied by $S$,
starting from the point that we let $row(S)$ denote the highest row 
containing $x$.  The fixed points are semi-standard tableaux of row shape
(those with all letters in row 1 and no multi-element cells).
\end{proof}

\begin{Cor}
The sum of the expansion coefficients $d_{\mu\lambda}(t)$ in
\begin{equation}
  H_\mu[X;t] = \sum_\lambda d_{\mu\lambda}(t) g_\lambda
\end{equation}
is $t^{n(\mu)}$.
\end{Cor}
\begin{proof}
There is only one semi-standard row shape tableau $T$ of weight $\mu$
and $\charge(T)=n(\mu)$.
\end{proof}

\bibliographystyle{alpha}
\bibliography{fields10}

\end{document}